\documentclass[oneside, 11pt,a4paper]{amsart}

\usepackage{stmaryrd}
\usepackage{subcaption}
\usepackage{graphicx} 
\usepackage{amssymb} \usepackage{amsfonts} \usepackage{amsmath}
\usepackage{amsthm} \usepackage{epsfig} 
\usepackage{color}
\usepackage[T1]{fontenc}
\usepackage{pinlabel}
\usepackage{amscd}
\usepackage[all]{xy}
\usepackage{pinlabel}
\usepackage{mathrsfs}  
\usepackage[most]{tcolorbox}
\usepackage{hyperref}
\usepackage{float}
\usepackage{tikz}
\usetikzlibrary{decorations.pathreplacing, backgrounds, calc}
\usepackage{geometry}
\usepackage{tikz-cd}
\usepackage[shortlabels]{enumitem}
\usepackage{bm}
\usepackage{xfrac}
\usepackage{cleveref}
\usepackage[]{mdframed}
\usepackage{lmodern}
\usepackage{mathtools}

\newtheorem{thm}{Theorem}[section]

\newtheorem{lem}[thm]{Lemma}

\newtheorem{corollary}[thm]{Corollary}
\newtheorem{prop}[thm]{Proposition}

\newtheorem{conj}{Conjecture}
\theoremstyle{definition}
\newtheorem{defn}[thm]{Definition}

\newtheorem{remark}[thm]{Remark}

\newcommand{\Z}{\ensuremath {\mathbb{Z}}}
\newcommand{\Q}{\ensuremath {\mathbb{Q}}}

\newtheoremstyle{namedd}
  {}{}              
  {\itshape}        
  {}                
  {\bfseries}       
  {.}               
  { }               
  {#1 #2\thmnote{ (#3)}} 
\theoremstyle{namedd}

\newtheorem{namedassumption}{Assumption}

\DeclareMathOperator{\id}{{id}}

\newcommand{\ds}[1]{{\color{blue}Diego: #1}}
\newcommand{\hz}[1]{{\color{red}Alice: #1}}

\newtcbtheorem{Chapter}{\bfseries Chapter}{enhanced,drop shadow={black!50!white},
  coltitle=black,
  top=0.3in,
  attach boxed title to top left=
  {xshift=1.5em,yshift=-\tcboxedtitleheight/2},
  boxed title style={size=small,colback=pink}
}{chapter}

\author{Alice Merz}
\author{Diego Santoro}

\address{HUN-REN Alfréd Rényi Institute of Mathematics, Reáltanoda utca 13-15, 1053, Budapest, Hungary}
\email{\href{mailto:alicemerz@renyi.hu}{alicemerz@renyi.hu}}
\email{\href{mailto:diego.santoro95@gmail.com}{diego.santoro95@gmail.com}}

\title{Zero-surgery characterises Whitehead doubles}

\begin{document}

\begin{abstract}
We show that zero is a characterising slope for the untwisted Whitehead double of any knot $K$. More generally, this holds for all $n$-twisted Whitehead doubles of $K$, with the assumption that $n\ne pq$ if $K$ is a non-trivial $(p,q)$-torus knot or $(p,q)$-cable. In particular, zero is a  characterising slope for all twist knots.
\end{abstract}

\maketitle
\section{Introduction}
A slope \(r\in \Q\) is said to be a \emph{characterising slope} for a knot \(K\subset S^3\) if $K$ is uniquely determined by the orientation-preserving diffeomorphism type of its $r$-surgery, which we denote $S^3_{r}(K)$. In other words, $r$ is a characterising slope for $K$ if $S^3_{r}(K)\cong S^3_{r}(K')$ implies that $K$ and $K'$ are isotopic. 

Every slope is characterising for the unknot \cite{KMOS07} as well as for the trefoil and the figure-eight knot \cite{OSDehnsur19}, and in general every knot $K$ has infinitely many characterising slopes \cite{Lac19}. On the other hand, there are examples of knots with non-characterising slopes, such as the hyperbolic two-bridge knot $8_6$, for which every integer slope is non-characterising. More generally, for each rational $r\in \Q$, there exists a pair of distinct knots with orientation-preserving diffeomorphic $r$-surgeries \cite{HPW25}.

In this paper we focus on the case of $0$-surgeries. This case has interesting relations with $4$-dimensional aspects of knot theory. It is well known that a knot $K$ is slice if and only if the trace of $K$ smoothly embeds in $S^4$ (this fact is sometimes referred to as the \emph{trace embedding lemma}, see for example \cite{KM78}), and this allows to certify the sliceness (or non-sliceness) status of a knot $K$ by finding a distinct knot $K'$ with the same trace whose status is known. Famously, this technique was used to show that the Conway knot is not slice \cite{Pic20}.
A problem that has received considerable attention is the following conjecture, appearing as Problem~1.39 in \cite{kirby1997problems} and Problem~1.40(b) in the K3 problem list \cite{K3}.  

\begin{conj}\label{conj}
The (untwisted) Whitehead double of a knot $K$ is slice if and only if $K$ is slice.
\end{conj}

A possible way to disprove this conjecture could be to look for slice knots having the same trace of the Whitehead double of a non-slice knot. Since $S^3_0(K)$ is the boundary of the trace of $K$, the main theorem of this paper shows that this strategy cannot work\footnote{See \cite{kegel2026search}, for related considerations on the (im-)possibility of applying the specific trick Piccirillo uses in \cite{Pic20}.}.

\begin{thm}\label{thm: main theorem}
    Let $W_K$ be the untwisted positive Whitehead double of \(K\subset S^3\). Then zero is a characterising slope for $W_K$.
\end{thm}
Note that Theorem~\ref{thm: main theorem} implies, by taking mirrors, that zero is a characterising slope for (untwisted) negative Whitehead doubles as well.  Characterising slopes for Whitehead doubles were studied in \cite{Wak}, with a focus on slopes $r=\sfrac{p}{q}$ for $q$ large enough.

Indeed, Theorem \ref{thm: main theorem} follows from a more general detection result for twisted Whitehead doubles. For $n\in \mathbb{Z}$, and $K\subset S^3$, we denote by $W^n_{K}$ the $n$-twisted positive Whitehead double of $K$.

\begin{thm}\label{thm: main thm gen}
Assume $K\subset S^3$ is not a $(p,q)$-torus knot nor a $(p,q)$-cable knot. Then, for each $n\in \Z$, zero is a characterising slope for $W^n_{K}$. If $K$ is a $(p,q)$-torus knot or a  $(p,q)$-cable, then the same holds when $n\ne pq$.
\end{thm}

\begin{remark}\label{rem: conventions}
In this paper, we follow the convention that the $(p,q)$-cable of a knot $J$ is the one that lies on the boundary of a tubular neighbourhood of $J$ and winds $p$ times around the canonical longitude of $J$ and $q$ times around the meridian of $J$. Moreover, we will always assume that cables are \emph{non-trivial}, that is to say that the $(p,q)$-cable of a knot $J$ is never an unlink or equal to $J$. Analogously, we will assume that a $(p,q)$-torus knot is not the unknot.
\end{remark}

By taking $K$ to be the unknot in Theorem~\ref{thm: main thm gen}, we have the following corollary.

\begin{corollary}
Zero is a characterising slope for all twist knots.
\end{corollary}

Very few cases of knots for which zero is a characterising slope are known, and whether this was true for the unknot was the content of the celebrated Property R conjecture, proved by Gabai in \cite{Gabaifol3}. More precisely, Gabai showed that $0$-surgery on $K$ detects the genus of $K$ and whether $K$ is fibered or not, and this implies that zero is a characterising slope for the trefoils and the figure-eight knot as well\footnote{To the best of the authors' knowledge, these knots and the knot $5_2$, studied in \cite{BStorusknot}, were the only twist knots for which it was known that zero is a characterising slope. We point out that our result does not provide an independent proof of this fact for the figure-eight and trefoil knots, since our proofs make use of Gabai's theorems.}.

For many years, these have been the only known examples, until Baldwin--Sivek recently  found the first infinite family of knots for which zero is a characterising slope \cite{BSinfinite}, consisting of the genus one knots whose knot Floer homology is two-dimensional in the top Alexander grading. Their proof rests on Heegaard Floer theoretic techniques, and in particular on their classification of this family of knots \cite{BSnonfib}. By contrast, our proof uses classical techniques in $3$-manifold topology.
\newline

In the course of the proof of Theorem \ref{thm: main thm gen}, we obtain a result of independent interest, that we report here. We refer to Section~\ref{defn: banding} for the definition of (oriented) banding of a link; see Definition~\ref{defn: banding}. Loosely speaking, this notion generalises that of band sum to the case of non-split links.  

\begin{thm}\label{thm: band in annulus}
   Let $L_K$ be the $(2,2n)$-cable of a knot $K$, oriented as the boundary of a cabling annulus. Let $\widetilde{K}$ be the knot obtained by an oriented banding of $L_K$ along a band $b$. Then $\widetilde{K}$ is the unknot if and only if $b$ can be isotoped in a cabling annulus for $L_K$.
\end{thm}

We notice that the case when $L_K$ is the unlink follows from \cite{gabai1987genus}. In fact, Gabai shows that the genus of a knot is superadditive under band sum, and it is additive exactly when a minimal genus Seifert surface for the band sum is obtained by band summing two minimal genus Seifert surfaces for the knots. Our Theorem \ref{thm: band in annulus} can be thought as a partial analogous for a family of links that is not split. Indeed, our proof is inspired by the one in \cite{gabai1987genus}.
\newline

Unless otherwise stated, diffeomorphisms between oriented manifolds will be assumed to be orientation preserving. 

\subsection*{Overview of the proof.} The outline of the proof of Theorem \ref{thm: main thm gen} is the following. Without loss of generality, we can assume that $W^n_K$ is a non-trivial knot. Suppose that $K'$  has $0$-surgery diffeomorphic to that of $W^n_K$. Then work of Gabai shows that $K'$ has genus one, and, after a study the JSJ decomposition of \(S^3_0(W^n_K)\), we show that any diffeomorphism between the surgeries must identify a capped minimal genus surface $\widehat{\Sigma}_{K'}$ for $K'$ with that of \(W^n_K\); see Proposition~\ref{prop: Sigma cap to F}. This implies that the manifold obtained by cutting the exterior \(E_{K'}\) of $K'$ along \(\Sigma_{K'}\) can be identified with the exterior of the $(2,2n)$-cable $L_K$ of $K$ after the removal of a properly embedded arc $\alpha$. Moreover, we show that there is a band $b$ with core $\alpha$ such that a banding of $L_K$ along $b$ yields an unknot $\widetilde{K}$ with the property that $K'$ is identified with the meridian of the band $b$ in $S^3_1(\widetilde{K})\cong S^3$. Finally, we use Theorem \ref{thm: band in annulus}, to prove that $K'$ is isotopic to $W^n_K$.

\subsection*{Structure of the paper.} 
In Section \ref{sec: lemmas} we prove some lemmas on the JSJ decomposition of the $0$-surgery on $W^n_K$, as well as Proposition~\ref{prop: Sigma cap to F}.
In Section~\ref{sec: bandings}, we focus on bandings of cable links, and prove Theorem \ref{thm: band in annulus}. Finally, Section~\ref{sec: main thm} is devoted to the  proof of Theorem~\ref{thm: main thm gen}.

\subsection*{Acknowledgements}
The authors started this project while participating to the PCMI Reseach Program ``Knotted Surfaces in Four-Manifolds'', where Conjecture~\ref{conj} was presented by Ivan So during a lecture. They would like to thank Ivan So for his lecture, the organisers of the research program for organising such a great event, and the IAS/Park City Mathematics Institute for providing an excellent research environment.

The authors would also like to thank Cameron Gordon for his interest in this project and for the stimulating conversations.
They thank Alessio Di Prisa for suggesting the reference \cite{Budney2005JSJdecompositionsOK}, and Marco Golla and Marco Marengon for providing valuable feedback during the final stages of this project. In particular, they thank Marco Marengon for the discussion that led to the proof of Lemma~\ref{lem: surfaces in solid torus}. Finally, they also would like to thank Marc Kegel for helpful comments on a final draft of the paper.

\subsection*{Individual grant support}
AM was supported by the grant Excellence 151337, and DS was supported by the ERC Advanced Grant KnotSurf4d.

\subsection*{AI statement}
 The use of AI and LLMs during this project was limited to literature research. AI and LLM tools were not used during the writing process, nor suggested any of the ideas appearing in this paper.

\section{Some preliminary lemmas on JSJ decompositions}\label{sec: lemmas}
In this section we collect some preliminary lemmas on the JSJ decompositions of the manifolds that will appear in the proof of the main theorem, with the primary goal of proving Proposition \ref{prop: Sigma cap to F}, which will be the starting point for the proof of Theorem \ref{thm: main thm gen}.

\begin{figure}[]
    \centering
\includegraphics[width=0.25\textwidth]{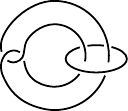}
    \caption{The positive Whitehead link.}
    \label{Whitehead link}
\end{figure}
Recall that a link \(L\subset S^3\) is \emph{fibered} if there exists a Seifert surface $\Sigma$ for $L$, an orientation preserving homeomorphism $h$ of $\Sigma$ fixing $\partial \Sigma$ pointwise, and an orientation preserving homeomorphism
$$
\chi: E_L \rightarrow M_h,
$$
where $E_L$ is the exterior of $L$ and $M_h$ is the mapping torus of $(\Sigma,h)$, so that 
\begin{itemize}
    \item $\chi_{|\Sigma}$ is the inclusion $\Sigma\subset M_h$;
    \item $\chi(m_i)=\{x_i\}\times [0,1]$, where $m_i$ is a meridian for the $i$-th component of $L$ and $x_i$ is a point \(\partial \Sigma\).
\end{itemize}

We denote by $W$ the positive Whitehead link, depicted in Figure~\ref{Whitehead link}, and recall the following fact.

\begin{lem}\label{lem: whitehead is fibered}
$W$ is a fibered link with fiber a twice-punctured torus $\Sigma_W$ and monodromy given by $h=\tau_0\tau_1\tau^{-1}_2$, where $\tau_i$ denotes the right-handed Dehn twist along the curve $\gamma_i$, shown in Figure~\ref{Monodromy of Whitehead}.
\end{lem}
\begin{proof}
See for example the discussion at the beginning of \cite[Section~3.2]{SanWhit}    
\end{proof}
We denote by $S^3_{0,\bullet}(W)$ the manifold obtained by performing $0$-framed Dehn surgery on one of the components of $W$ while drilling the other. Since the components of $W$ can be exchanged with an isotopy, this manifold does not depend on the choice of which component we drill and surger.

Fix once and for all an orientation of $S^1$, and let $P$ denote the pair of pants. Let us fix an arbitrary orientation on $P$ and consider the manifold $P\times S^1$, endowed with the product orientation.  Denote the boundary components of $P$ by $\partial_0, \partial_1, \partial_2$, oriented as the boundary of $P$, and fix peripheral bases for $\partial_0 \times S^1$ and $\partial_1 \times S^1$ given by $(\partial_i, \{p_i\}\times S^1)$, with $p_i\in \partial_i$ for $i=0,1$. 

More generally, if $M_h$ is the mapping torus of an orientation preserving homeomorphism $h$ of an oriented surface $\Sigma$ fixing $\partial \Sigma$ pointwise, and $\partial_i \Sigma$ is a boundary component of $\Sigma$ with the induced orientation, we refer to the peripheral bases given by $(\partial_i \Sigma, \{p_i\}\times S^1)$, where $p_i\in \partial_i \Sigma$, as the \emph{standard} peripheral bases.

\begin{figure}[]
    \centering
\includegraphics[width=0.35\textwidth]{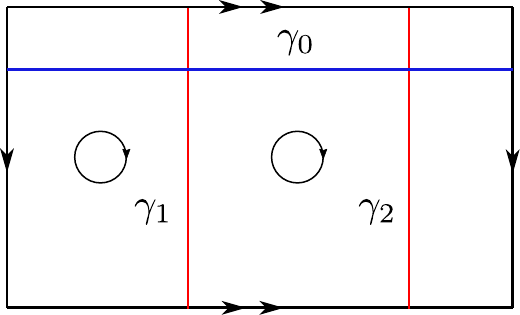}
    \caption{A twice-punctured torus, obtained by identifying the boundary of the square as indicated by the arrows, and the curves $\gamma_i$, for $i=0,1,2$.}
    \label{Monodromy of Whitehead}
\end{figure}
\begin{lem} \label{lem: gluing map}
The manifold $S^3_{0,\bullet}(W)$ is  diffeomorphic to \((P\times S^1)/_Q\), where $Q$ is the diffeomorphism between $\partial_0\times S^1$ and $\partial_1\times S^1$ defined by the matrix $\left(\begin{smallmatrix}
-1 & -1 \\
0 & 1
\end{smallmatrix}\right)$ in the standard peripheral bases.
\end{lem}
\begin{proof}
Since the components of the Whitehead link have linking number zero, the longitudes defined by a Seifert surface for $W$ coincide with the canonical longitudes of the link components. This implies that the  manifold $S^3_{0,\bullet}(W)$, fibers over the circle, with fiber $\Sigma$ obtained by capping one of the boundary components of the fiber surface $\Sigma_W$ with a disc, and monodromy $h'$ obtained by extending the monodromy of $W$. By the explicit description given in Lemma~\ref{lem: whitehead is fibered}, we deduce that the monodromy $h'$ is a single right-handed Dehn twist along a non-separating curve $\gamma_0\subset \Sigma$. By cutting $S^3_{0,\bullet}(W)$ along the torus spanned by $\gamma_0$ via the monodromy, we obtain a manifold $M$ diffeomorphic to $P\times S^1$. More precisely, if we denote with $\partial_0$ and $\partial_1$ the boundary components of $P$ coming from the boundary of a tubular neighbourhood of $\gamma_0$ in $\Sigma$ and think of $\gamma_0$ as a boundary parallel curve in $P$, then 
\[
M=\frac{P\times [0,1]}{(\tau_0(x),0)\sim (x,1)}.
\]
By construction, in the standard peripheral bases of $\partial_0 P\times S^1$ and $\partial_1 P\times S^1$ of $M$, the homological action of the gluing map $f$ needed to recover $S^3_{0,\bullet}(W)$ is described by $\left(\begin{smallmatrix}
-1 & 0 \\
0 & 1
\end{smallmatrix}\right).$
The manifold $P\times S^1$ is diffeomorphic to $M$ via a diffeomorphism supported in a collar of $\partial_0\times S^1$. More precisely, let $A=S^1\times [0,1]$ be a collar neighbourhood of $\partial_0\subset P$, with $S^1\times \{0\}=\partial_0$, containing the curve $\gamma_0$. Since $\tau_0^{-1}$ is isotopic to the identity of $A$ relative to $S^1\times \{1\}$, we can find a level preserving diffeomorphism $H$ of $A\times [0,1] $
so that
\begin{itemize}
\item $H$ restricts to the identity on $A\times \{0\}\cup S^1\times \{1\}\times [0,1]$;
\item $H$ is $\tau_0^{-1}$ on $A\times \{1\}$;
\item $H$ restricts to a left-handed Dehn twist on $S^1\times \{1\}\times [0,1]$.
\end{itemize}
By plugging the above diffeomorphism in a collar neighbourhood of $\partial_0\times S^1$, we obtain a diffeomorphism $g$ from $P\times S^1$ to $M$ which in the standard peripheral bases is given by the matrix  $\left(\begin{smallmatrix}
1 & 1 \\
0 & 1
\end{smallmatrix}\right)$ when restricted on $\partial_0\times S^1$, and the identity on the other boundary components. In conclusion, the manifold $S^3_{0,\bullet}(W)$ is obtained by identifying $P\times S^1$ with $M$ via $g$ and then gluing the boundary components via the map $f$, and this proves the desired result.
\end{proof}

We recall that, for $n\in \Z$, we denote by $W^n_{K}$ the $n$-twisted positive Whitehead double of $K$, and we refer the reader to the introduction for our conventions on torus knots and cable links. 

\begin{namedassumption}\label{assum 1}
Unless otherwise stated, for the rest of this section we will assume that $K\subset S^3$ is either a non-trivial knot or that $K$ is the unknot and $n\neq -1, 0, 1$.
Moreover, if $K$ is the $(p,q)$-torus knot, or the $(p,q)$-cable of a non-trivial knot, we further assume  $pq \ne n$. Finally, we suppose that  $K'$ is a knot such that $S^3_0(K') \cong S^3_0(W_K^n)$. 
\end{namedassumption}

We observe that part of the assumptions are made to rule out the case of $W^n_K$ being a trefoil knot, the unknot or the figure-eight knot. In fact, this happens exactly when $K$ is the unknot and $n=1, 0, -1$ respectively. 
\newline

We denote by $F$ the properly embedded torus given by the image of $\partial_0 P\times S^1$ in $S^3_{0,\bullet}(W)$, and for any non-split link $L\subset S^3$, we denote by $E_L$ the exterior of $L$.

We recall here the JSJ decomposition theorem by Jaco--Shalen and Johannson, and a result on essential annuli and Seifert fibrations.

\begin{thm}[JSJ decomposition theorem \cite{JacoShalen, Johannson}]
For every compact orientable and irreducible $3$-manifold $M$ there exists a collection of finitely many disjoint incompressible tori 
such that each component formed by cutting along these tori is either atoroidal or a Seifert manifold. Moreover, a minimal such collection is unique up to isotopy. 
\end{thm}

We will refer to such a minimal collection of tori as the \emph{JSJ tori} of $M$, and to the components of their complement as the \emph{JSJ pieces} of $M$. The JSJ piece of the exterior of a knot $J$ containing $\partial E_J$ will be called the \emph{outermost piece}.

\begin{thm}\label{thm: wald and Hatc}
    Let $M$ be a Seifert fibered manifold with non-empty boundary, and
    suppose that $M$ is not homeomorphic to a solid torus, to $S^1 \times S^1 \times I  $ or to the twisted $I$-bundle over the Klein bottle. Then the Seifert fibration on $M$ is unique up to isotopy, and every essential annulus in $M$ can be isotoped to be vertical.
\end{thm}
\begin{proof}
The statement on the uniqueness of the Seifert fibration on $M$ up to isotopy is a direct consequence of \cite[Satz 10.1]{WALDHAUSEN1967}; see also \cite[Theorem VI.18]{jaco1980lectures}.
As a consequence of \cite[Lemma 1.16]{hatcher2007notes}, all essential annuli can be isotoped to be vertical with respect to this Seifert fibration.
\end{proof}

\begin{lem} \label{lem: JSJ}
Let $K\subset S^3$ be a non-trivial knot satisfying Assumption \ref{assum 1}, and let $\mathcal{T}=\{T_1,\dots, T_k\}$ be the JSJ tori of the JSJ decomposition of $E_K$. Then the manifold $S^3_0(W_K^n)$ has a JSJ decomposition given by $\mathcal{T}\cup \{\partial E_K, F\}$.
\end{lem}

\begin{proof}
We first prove that all tori in $\mathcal{T}\cup \partial E_K\cup F$ are incompressible in $S^3_0(W_K^n)$, starting with $F$. First recall that \cite[Theorem~3.1]{Gabaifol3} and the fact that $W_K^n$ is a non-trivial knot imply that $S^3_0(W_K^n)$ is irreducible. In particular, since $F$ is non-trivial in the second homology group of $S^3_0(W_K^n)$, it is incompressible, otherwise it would be homologous to a sphere, and hence separating. 

By construction, cutting $S^3_0(W_K^n)$ along $\partial E_K$ yields $E_K\cup S^3_{0,\bullet}(W)$. The torus $\partial E_K$ is incompressible in $E_K$ since $K$ is non-trivial, and incompressible in $S^3_{0,\bullet}(W)$, since a compressing disc, via an innermost argument, can be made disjoint from $F$, and hence would imply the existence of a compressing disc inside $P\times S^1$, which is not possible. This implies that $\partial E_K$ is incompressible in $S^3_0(W_K^n)$.

Regarding the tori in $\mathcal{T}$, if there exists a compressing disc in $S^3_0(W_K^n)$, then by an innermost argument again, we can suppose that it is disjoint from $\partial E_K$. However, this would imply compressibility in $E_K$, which contradicts our hypotheses.

 By construction, cutting $S^3_0(W_K^n)$ along the tori in this collection produces a union of Seifert fibered manifolds and atoroidal manifolds. In order to prove that the collection $\mathcal{T}\cup \{\partial E_K, F\}$ defines a JSJ decomposition, we need to show that this set is minimal.
First of all, notice that $S^3_0(W_K^n)$ is not a torus semi-bundle since its first homology group does not contain any torsion (we refer to \cite{Martelli2016AnIT} for a definition of torus semi-bundles). Moreover, by \cite[Corollary~8.23]{Gabaifol3}, the manifold $S^3_0(W_K^n)$ is a torus bundle if and only if $W_K^n$ is, up to mirror, the trefoil or the figure-eight knot. Since $K$ is a non-trivial knot by hypothesis, we conclude that $S^3_0(W_K^n)$ is not a torus bundle.

Therefore the set of incompressible tori $\mathcal{T} \cup \{\partial E_K , F\}$ is not minimal if and only if, when we cut $S^3_0(W_K^n)$ along them, two adjacent pieces are Seifert fibered and their fibers match up under the gluing (see for example the proof of \cite[Proposition~11.5.4]{Martelli2016AnIT}).
 As a consequence, the tori in $\mathcal{T}$ must be JSJ tori for $S^3_0(W_K^n)$. By Theorem \ref{thm: wald and Hatc},
 the manifold $P \times S^1$ has a unique Seifert fibration up to isotopy, and its fibers are $\{x\} \times S^1$, $x \in P$. Since the fibers are not preserved by the gluing map in Lemma \ref{lem: gluing map}, it follows that $S^3_{0,\bullet}(W)$ is not Seifert fibered. Therefore, $F$ is a JSJ torus for $S^3_0(W_K^n).$

 We are left with checking that $\partial E_K$ is a JSJ torus.
Of course, if the outermost piece $N$ of the JSJ decomposition of $E_K$ is not Seifert fibered then $\partial E_K$ is a JSJ torus.
We therefore just need to check the case where $N$ is Seifert fibered. In that case, by \cite{Budney2005JSJdecompositionsOK} (see also \cite[Theorem 2.6]{Wak}), there are three possibilities:
\begin{itemize}
\item $K$ is the $(p,q)$-torus knot;
\item $K$ is the $(p,q)$-cable of a non-trivial knot $J$;
\item $K$ is a composite knot.
\end{itemize}
In the first and second case, since $K$ is non-trivial, $N$ has a unique Seifert fibration and its regular fibers have slope $pq \neq 0$ on $\partial E_K$. 
Lastly, if $K$ is composite, the regular fibers of the outermost piece have meridional slope on $\partial E_K$. 

Recall that the boundary component of $P\times S^1$ identified to $\partial E_K$ corresponds to the boundary of a tubular neighbourhood of one of the two components of the Whitehead link, and by the definition of fibered link given in Section \ref{sec: lemmas}, the regular fibers $\{pt\}\times S^1$ corresponds to meridians of this component. Since we are considering the $n$-twisted Whitehead double of $K$, the gluing map identifies these meridians to the $n$-longitudes of $K$, and since $K$ satisfies Assumption \ref{assum 1}, they are not identified to the regular fibers of $N$. This proves that $\partial E_K$ is a JSJ torus of $S^3_0(W_K^n)$.
\end{proof}

\begin{lem} \label{lem: JSJ banale}
Let $n\neq -1, 0, 1$ and $K\subset S^3$ be the unknot. Then $F$ is the unique JSJ torus of $S^3_0(W_K^n)$, and the unique JSJ piece is Seifert fibered and atoroidal.
\end{lem}

\begin{proof}
 First notice that also in this case $S^3_0(W^n_K)$ is irreducible and $F$ is incompressible. 
 Cutting $S^3_0(W^n_K)$ along $F \cup \partial E_K$ yields the disjoint union of $E_K$ and $P \times S^1$, and the fibers of the Seifert fibration of $P\times S^1$ are identified with simple closed curves of slope $n$ in $\partial E_K$.

 Since $E_K$ is a solid torus, it admits a Seifert fibration with fibers of slope $ {n}$ on $\partial E_K$. This shows that $S^3_0(W^n_K)$ cut along $F$ is a Seifert fibered space over the annulus with one exceptional fiber of order $|n|$, and in particular is atoroidal.

 If this manifold is not $S^1 \times S^1 \times I$, this Seifert fibration is unique up to isotopy (see Theorem \ref{thm: wald and Hatc}). 
 The manifold $S^1\times S^1 \times I$ admits only one Seifert fibration up to diffeomorphism \cite[Theorem 2.3]{hatcher2007notes}, and has no exceptional fibers. Hence, if $n\neq \pm 1$, $S^3_0(W^n_K)$ cut along $F$ is not $S^1 \times S^1 \times I$.
 Moreover, since $n \neq \pm 1$, the manifold $S^3_0(W^n_K)$ is not a torus (semi-)bundle (see the proof of Lemma \ref{lem: JSJ}), and therefore the same argument as in the proof of Lemma \ref{lem: JSJ} allows us to show that $F$ is a JSJ torus.
\end{proof}

\begin{lem} \label{lem: F unique}
    If $S$ is an incompressible, non-separating torus in $S^3_0(W_K^n)$, then $S$ is isotopic to $F$.
\end{lem}

\begin{proof}
We start by showing that $S$ can be assumed to be disjoint from the JSJ tori of $S^3_0(W^n_K)$. The argument follows essentially the proof of uniqueness of the JSJ decomposition.

Recall that by Lemma \ref{lem: JSJ} and Lemma \ref{lem: JSJ banale} we know the JSJ decomposition of $S^3_0(W^n_K)$.
After performing an isotopy, we can suppose that $S$ is transverse to all the JSJ tori 
and that the intersection of $S$ with the JSJ tori has a minimal number of connected components. We assume that this number is positive, since otherwise $S$ is already disjoint from the JSJ tori.
Since the JSJ pieces are irreducible and the JSJ tori are incompressible, we can suppose that $S$ is cut by the JSJ tori into annuli, and that each of these annuli is essential in its JSJ piece. The atoroidal pieces that contain essential annuli are Seifert fibered (see \cite[Lemma 1.18]{hatcher2007notes}), and hence $S$ must intersect only Seifert fibered JSJ pieces. As a consequence of Theorem \ref{thm: wald and Hatc} and our hypotheses, the Seifert fibered JSJ pieces of $S^3_0(W_K^n)$ all admit a unique Seifert fibration up to isotopy, and essential annuli can be isotoped to be vertical in their Seifert fibered pieces. This means that if $T$ is a JSJ torus with $T \cap S \neq \emptyset$, then the Seifert fibrations on the Seifert pieces containing $T$ match along $T$. But this would lead to a contradiction, since in that case $T$ could be removed from the JSJ tori.

Therefore we can suppose that $S$ is disjoint from the JSJ tori, and hence contained in one of the JSJ pieces. 
If $K$ is the unknot, by Assumption \ref{assum 1} we know that $n \neq 0,1,-1$, and that by Lemma \ref{lem: JSJ banale} the unique JSJ piece is Seifert fibered and atoroidal. Since $S$ is incompressible in $S^3_0(W^n_K)$, we deduce that it is isotopic to $F$. 

We will therefore suppose that $K$ is non-trivial.
Since every torus in $S^3$ is separating and $S$ is not, $S$ cannot be contained in $E_K\subset S^3$, meaning that $S$ must be contained in the JSJ piece $P \times S^1$. Up to isotopy, the torus $S$ is vertical in $P \times S^1 $ (see \cite[Proposition 1.12]{hatcher2007notes}), and hence isotopic to one of its boundary components. In $S^3_0 (W_K^n)$ the boundary components of $P \times S^1$ get sent to $\partial E_K$ and $F$. Again, since $S$ is not separating, it cannot be isotopic to $\partial E_K$, and so it must be isotopic to $F$.
\end{proof}

Although not strictly necessary for our purposes, the previous proposition holds also when $W^n_K$ is the trefoil or the figure-eight knot. For completeness we provide a proof.

\begin{lem}\label{lem: fig8 and trefoil}
Assume that $K$ is the unknot and that $n=\pm 1$. If $S$ is an incompressible, non-separating torus in $S^3_0(W^n_K)$, then $S$ is isotopic to $F$.
\end{lem}

\begin{proof}
In this case, $W^n_K$ is the trefoil or the figure-eight knot, and its $0$-surgery fibers over the circle, with $F$ as a fiber. Every non-separating torus in $S^3_0(W^n_K)$ has a dual curve, and therefore it is a generator of
the second homology group. In particular, it must be homologous, and hence isotopic to $F$ (see, e.g. \cite[Chapter 1.4]{EisNeu}). 
\end{proof}

Given a knot $J\subset S^3$, with Seifert surface $\Sigma_J$, we denote by $\widehat{\Sigma}_J$ the closed surface in $S^3_0(J)$ obtained by capping $\Sigma_J$ with a meridional disc of the solid torus. 

We conclude the section with the following proposition, which will be the starting point for the proof of the main result of the paper.

\begin{prop} \label{prop: Sigma cap to F}
Let $K$ and $K'$ be knots in $S^3$ satisfying Assumption \ref{assum 1}. Then the $3$-genus of $K'$ is one, and if $\Sigma$ is a Seifert surface of genus one for $K'$, the diffeomorphism $S^3_0(K') \cong S^3_0(W_K^n)$ sends the capped Seifert surface $ \widehat{\Sigma}$ to a surface isotopic to $F$. The same holds if $W_K^n$ is the trefoil or the figure-eight knot.
\end{prop}
\begin{proof}
Let $\Sigma$ be a minimal genus Seifert surface for $K'$.
As consequence of \cite[Theorem~3.1]{Gabaifol3}, the surface $\widehat{\Sigma} $ has minimal genus among the surfaces representing a generator of $H_2(S^3_0(K');\Z) \cong \Z$. In particular, $\widehat{\Sigma}$ is incompressible in $S^3_0(K')$, and if two knots have diffeomorphic $0$-surgeries, then they have the same $3$-genus. Since the $n$-twisted Whitehead double of a knot has genus one whenever $n \neq 0$ or $K$ is non-trivial, we deduce that $K'$ has genus one. This implies that $\widehat{\Sigma}$ is an incompressible torus in $S^3_{0}(K')$, and by Lemma \ref{lem: F unique}, or Lemma \ref{lem: fig8 and trefoil} in the case of the trefoil or the figure-eight knot,  the diffeomorphism $S^3_0(K') \cong S^3_0(W_K^n)$ maps it to a surface isotopic to $F$.
\end{proof}

\section{Bandings of $(2,2n)$-cable links}\label{sec: bandings}
In this section, we focus on links $L_K$ that are $(2,2n)$-cables of a knot $K\subset S^3$, with the goal of proving Theorem \ref{thm: band in annulus}, which characterises when the result of a \emph{banding} of $L_K$ is the trivial knot. 

We recall the definition of a banding of a link $L$ (see for example \cite{Munoz}). 

\begin{defn}\label{defn: banding}
Let $L$ be a link in $S^3$, and let $b$ be an embedding of $I\times I$ into $S^3$ with $b(I\times I)\cap L=b(I\times \partial I)$. The link $L'=\left(L\setminus b(I\times \partial I)\right)\cup b(\partial I\times I)$ is called a \emph{banding} of $L$. If $L$ is oriented, an \emph{oriented banding} of $L$ is a banding that respects the orientations of the components of $L$. 
\end{defn}

\begin{remark}
    Note that if the link $L$ is a split link with two components, and the band intersects both components, then a banding of $L$ corresponds to a band sum of the components of $L$.
\end{remark}

In what follows, if $L_K$ is the $(2,2n)$-cable of a knot $K$, we will call a \emph{cabling annulus} an annulus $A\subset E_{L_K}$ such that $\partial A= L_K$ and such that removing a tubular neighbourhood of $A$ from $E_{L_K}$ yields $E_K$.
We start by classifying incompressible annuli in $E_{L_K}$. Recall from Remark \ref{rem: conventions} that under our conventions $L_K$ is never the unlink.

\begin{lem} \label{lem: unique annulus}
Let $L_K$ be the $(2,2n)$-cable of a knot $K$, and assume that $L_K$ is not the Hopf link. If $K$ is the $(p,q)$-torus knot or the $(p,q)$-cable of a non-trivial knot $J$, further assume that $n\ne pq$. Then, there exists a unique incompressible annulus $A$ in $E_{L_K}$, and it is a cabling annulus. If $L_K$ is the Hopf link, then there exist exactly two incompressible annuli $A_1$ and $A_2$ in $E_{L_K}$ with the property that $\partial A_i = L_{K}$, and they are cabling annuli.
\end{lem}

\begin{proof}
If $K$ is non-trivial, arguing as in the proof of Lemma~\ref{lem: JSJ}, one shows that the JSJ decomposition of $E_{L_K}$ has as outermost piece $P \times S^1$. The annulus $A$ must be vertical in $P \times S^1$ and therefore it coincides with the unique cabling annulus $c \times S^1$, where $c$ is an arc in $P$ connecting the two boundary components $\partial_0$ and $\partial_1$. Clearly, cutting $E_{L_K}$ along $A$ yields $E_K$.  Suppose now that $K$ is trivial and $n \neq \pm 1$. The manifold $E_{L_K}$ admits a Seifert fibering over the annulus with one exceptional fiber of order $|n|$ and $A$ must be vertical by Theorem \ref{thm: wald and Hatc}. Since all vertical annuli in $E_{L_K}$ are isotopic, we conclude.

We now consider the case of $K$ trivial and $n=\pm 1$, i.e. when $L_K$ is the Hopf link. In this case, $E_{L_K}\cong T\times I$, and by \cite[Lemma 1.16]{hatcher2007notes}, every essential annulus can be made vertical, up to changing the Seifert fibration. Since every Seifert fibration of $T\times I$ is given by crossing a fibration in circles of $T$ with $I$, this implies that there are at most two non-isotopic essential annuli in $E_{L_K}$ with the required boundary: those with slope $+1$ and $-1$ respectively. Since they have different slopes on the boundary, they are not isotopic.
\end{proof}

\begin{lem} \label{lem: every annulus is cabling}
    Let $L_K$ be the $(2,2n)$-cable of a knot $K$. Let $A$ be an incompressible annulus in $E_{L_K}$ with $\partial A=L_K$. Then $A$ is a cabling annulus for $L_K$. 
\end{lem}

\begin{proof}
By Lemma \ref{lem: unique annulus} we only need to consider the case of $K$ being the $(p,q)$-torus knot or the $(p,q)$-cable of a non-trivial knot $J$ and $n=pq$. In the first case, $L_K$ is the $(2p,2q)$ torus link and $E_{L_K}$ is a Seifert fibered space over the annulus with two exceptional fibers of order $|p|$ and $|q|$. In the second case, the outermost piece of the JSJ decomposition of $E_{L_K}$ is Seifert fibered with one exceptional fiber of order $|q|$. In both cases there are many isotopy classes of annuli intersecting both components of $\partial E_{L_K}$. However, by using Theorem \ref{thm: wald and Hatc} as in the previous proof, one sees that  these annuli must all be vertical, and that cutting $E_{L_K}$ along any of them always yields $E_K$. 
\end{proof}

Before proving Theorem \ref{thm: band in annulus} we need the following lemma.

\begin{lem}\label{lem: surfaces in solid torus}
Let $S$ be a genus one surface properly embedded in $S^1\times D^2$ such that $\partial S$ is equal to a meridian. Then $\partial S$ bounds a disc in $S^1\times D^2$ intersecting $S$ only in its boundary.    
\end{lem}
\begin{proof}
Let $D$ be a meridional disc of $S^1 \times D^2$. We can suppose that $D$ and $S$ are transverse, and that the number of connected components of their intersection is minimal. If $D$ and $S$ are disjoint, then $\partial D$ and $\partial S$ co-bound an annulus $A \subseteq \partial (S^1 \times D^2$), and $A \cup D$ is a disc whose intersection with $S$ is $\partial S$.

If $D$ and $S$ are not disjoint, then take a component $C \subseteq D \cap S$ that is innermost in $D$, and let $\Delta$ be the disc in $D$ bounded by $C$. 
Since the intersection between $S$ and $D$ is minimal, cutting $S$ along $D$ yields either a connected surface, or an annulus and a genus one surface. In the first case we can use two parallel copies of $\Delta $ to cap off the cut surface, while in the second case we can use $\Delta$ to cap off the annulus. Both procedures produce a disc. By slightly perturbing the disc, we can make it disjoint from the interior of $S$.
\end{proof}

We are now ready to prove the main result of this section. We denote by $K_0 \cup K_1$ the components of $L_K$, and recall from the statement of Theorem \ref{thm: band in annulus} that they are oriented as the boundary of a cabling annulus.

\begin{proof}[Proof of Theorem \ref{thm: band in annulus}]
 Of course, if the band $b$ can be isotoped inside a cabling annulus then $\widetilde{K}$ is an unknot. We consider the link $L=\widetilde{K}\cup J$, where $J$ denotes a meridian of the band $b$. Notice that the components of $L$ have linking number zero, 
 and hence $\widetilde{K}$ bounds some surface in the exterior of $J$. Let $S$ be such a surface of minimal genus. We can assume that $S$ intersects the disc $D$ bounded by $J$ in the tubular neighbourhood of $b$ only in one arc. In fact, since $S$ has minimal genus, it must be incompressible in the exterior of $J$ and one deduces, by an innermost argument, that up to isotopy $S$ must intersect $D$ only in one arc plus possibly some nested circles parallel to $J$. If there are such circles, then we can simply replace $J$ with the innermost one of them, and study a link isotopic to $L$, for which the minimal genus Seifert surface for $\widetilde{K}$ has the desired property.
 
We now determine the genus $g(S)$ of $S$.
First of all, notice that $g(S)>0$. In fact, if $S$ is a disc, then cutting $S$ along the arc of intersection with $D$ would produce two discs in $S^3$ whose boundaries are $K_0$ and $K_1$. However, this is not possible since $L_K$ is not the unlink.

We now show that $g(S)<2$. Assume that this is not the case, and observe that $E_{L}$ is an irreducible and boundary-irreducible $3$-manifold. Irreducibility follows from the fact that $L$ is not a split link, otherwise $\widetilde{K}$ would have to bound a disc in the exterior of $J$, and we have already shown that this is not the case; compressibility of any of the two components of $\partial E_{L}$ would have the same consequence.
Since $E_{L}$ is irreducible and boundary-irreducible, by \cite[Corollary~2.4]{gabaifol2}, we know that the surface $S$ remains genus minimising in its homology class in all but at most one filling (that we call \emph{exceptional)} of $E_L$ along the boundary component corresponding to $J$. Since $\widetilde{K}$ is the trivial knot in $S^3$, we know that the exceptional filling must be the one given by $\infty$-surgery on $J$. However, we now show that also the filling given by $0$-surgery on $J$ must be exceptional, hence leading to a contradiction. In fact, in the $0$-surgery on $J$, we can isotope the band $b$ so that it becomes disjoint from a cabling annulus $A$ (see Figure \ref{fig:band slide}). 
\begin{figure}
    \centering
    \includegraphics[width=0.6\linewidth]{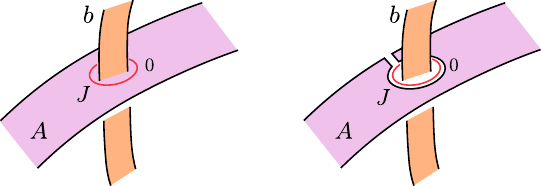}
    \caption{How to remove intersections between the band $b$ and the cabling annulus $A$ in the $0$-surgery along $J$.}
    \label{fig:band slide}
\end{figure}
Since $\widetilde{K}$ is obtained by oriented banding, $b\cup A$ is a genus one surface and, by checking its intersection with generators of the first homology group of the exterior of $\widetilde{K}$ in $S^3_0(J)$, one sees that it represents
the same homology class of $S$.

Summing up, we have that $S$ must have genus one. Using again that $L_K$ is not the unlink in $S^3$, we have that cutting $S$ along its intersection with $D$ produces an incompressible annulus $A'$ in the exterior of $L_K$, with $\partial A'=L_K$. Lemma \ref{lem: every annulus is cabling} implies that $A'$ is a cabling annulus for $L_K$, and $S$ is obtained by gluing the band $b$ to $A'$. In particular, the band and $A'$ are disjoint.

Since $\widetilde{K}$ is the unknot, by Lemma \ref{lem: surfaces in solid torus}, the knot $\widetilde{K}$ bounds a disc whose interior is disjoint from $S$, and in particular from the band $b$. The annulus obtained by gluing the band to this disc is a properly embedded incompressible annulus with boundary $L_K$, and therefore is isotopic in $E_{L_K}$ to a cabling annulus. In particular, the band $b$ can be isotoped in it, thus concluding the proof. 
\end{proof}

\section{Proof of the main theorem}\label{sec: main thm}
In this section we prove Theorem \ref{thm: main thm gen}.
Unless otherwise stated, we will have the following assumptions.

\begin{namedassumption}\label{assum 2}
We assume that either $n\neq 0$ or that $n=0$ and $K$ is a non-trivial knot in $S^3$, and that $K'$ is a knot such that $S^3_0(K') \cong S^3_0(W_K^n)$. 
Moreover, if $K$ is the $(p,q)$-torus knot, or the $(p,q)$-cable of a non-trivial knot, we further assume  $pq \ne n$. 
\end{namedassumption}

We also fix some notation. In this section, $\Sigma$ will always denote a genus one Seifert surface for $K'$, and $\widehat{\Sigma}$ will denote the surface in $S^3_0(K')$ obtained by capping $\Sigma$.

In the $0$-surgery on $K'$, let $\gamma$ be the \emph{dual curve} of $K'$, i.e., the core of the filled solid torus. 
Cutting $S^3_0(K')$ along the capped surface \(\widehat{\Sigma}\) yields a connected manifold with two boundary components, that we call \(N_{\widehat{\Sigma}}\).
Since \(\widehat{\Sigma}\) intersects $\gamma$ in one point, cutting $\gamma$ along \(\widehat{\Sigma}\) yields an arc $\alpha$ in \(N_{\widehat{\Sigma}}\), connecting the two boundary components. We will refer to $\alpha$ as the \emph{dual arc}.
Finally, we will denote by $M_{\Sigma}$ the complement in $S^3$ of a tubular neighbourhood for $\Sigma$. 

\begin{lem} \label{lem: cutting is cable}
    The manifold $N_{\widehat{\Sigma}}$ is diffeomorphic to the exterior of the {$(2,2n)$}-cable of $K$.
\end{lem}

\begin{proof}
By Proposition \ref{prop: Sigma cap to F}, $N_{\widehat{\Sigma}}$ is diffeomorphic to the manifold obtained by cutting $S^3_0(W_K^n)$ along $F$. This is the result of gluing $E_K$ to the $3$-manifold $P \times S^1$, so that if $\partial_0,\partial_1,\partial_2$ are the three boundary components of $P$ and $ p \in \partial_2$, then $\partial_2 \times \{x\}$ is identified to the meridian of $K$, while the regular fiber $ \{p\}  \times S^1$ is identified to the $n$-longitude of $K$.

On the other hand, the exterior of the $(2,2n)$-cable of $K$ is obtained by gluing to $E_K$ the manifold $X$ given by the exterior in the standard solid torus $\mathbb{T}$ of two unlinked parallel copies of the core. Clearly, $X$ is diffeomorphic to $P \times S^1$. This diffeomorphism identifies:
\begin{itemize}
\item the meridian of $\mathbb{T}$ with $\partial_2 \times \{x\} \subseteq \partial P \times S^1$;
\item the meridians of the two parallel copies of the core to $\partial_0 \times \{x\}$ and $ \partial_1 \times \{x\}$ respectively;
\item the standard longitude of $\mathbb{T}$ to $ \{p \} \times S^1$. 
\end{itemize}
The gluing map used to obtain the exterior of the $(2,2n)$-cable sends the meridian of this solid torus to the meridian of $K$, and sends the standard longitude of the solid torus to a $n$-longitude for $K$. Since this is the same gluing map used to obtain $N_{\widehat{\Sigma}}$, the two manifolds are diffeomorphic. 
\end{proof}
Let $L_K=K_0 \cup K_1 $ denote the $(2,2n)$-cable of $K$, and let, for $i=0,1$, $\widehat{\Sigma}_i$ be the boundary of the tubular neighbourhood of $K_i$. We orient $L_K$ so that $K_0$ and $K_1$ have parallel orientations, and denote by $m_{K_i}, l_{K_i} \subseteq \widehat{\Sigma}_i$ the oriented meridian and the $n$-framed longitude of $K_i$, for $i=0,1$. We will call $\widehat{\varphi}$ the diffeomorphism $\widehat{\Sigma}_0 \to \widehat{\Sigma}_1$ that is represented by the matrix $\left(\begin{smallmatrix}
    -1 & 1 \\ 0 & 1
\end{smallmatrix}\right)$ in the bases $(m_{K_i},l_{K_i})$. 
\begin{lem} \label{lem: glue NSigma back}
   The manifold $S^3_0(K')$ is diffeomorphic to the manifold obtained from $E_{L_K}$ by identifying $\widehat{\Sigma}_0$  to $\widehat{\Sigma}_1 $ via $\widehat{\varphi}$.
\end{lem}

\begin{proof}
    Since $S^3_{0}(W_K^n)$ is diffeomorphic to $S^3_{0}(K')$, the proof follows directly from Lemma~\ref{lem: cutting is cable} and Lemma~\ref{lem: gluing map} and from the fact that, up to reversing the orientation of both components of $L_K$ we have
    \begin{align*}
    &m_{K_i} =-\partial_i\\
    &l_{K_i} =\{p_i\}\times S^1,\quad p_i\in \partial_i 
    \end{align*}
    for $i=0,1$.
\end{proof}

We recall that $\alpha \subset N_{\widehat{\Sigma}}$ is the dual arc of $K'$, and we identify it with its image in $E_{L_K}$ via the diffeomorphism of Lemma \ref{lem: cutting is cable} (see Figure \ref{fig:handlebody}). 
Let $H_\alpha$ denote the genus two knotted handlebody obtained by taking the union of a tubular neighbourhood of $L_K$ and of a tubular neighbourhood of $\alpha$.

\begin{prop} \label{prop: MSigma handlebody}
    The manifold $M_\Sigma$ is diffeomorphic to the closure $\overline{S^3 \setminus H_{\alpha}}$ of $S^3 \setminus H_{\alpha}$. This diffeomorphism sends a canonical longitude of $K'$ in $ M_\Sigma$ into a meridian of $\alpha$.
\end{prop}

\begin{proof}
The manifold $M_{\Sigma}$ is obtained by cutting $E_{K'}$ along the Seifert surface $\Sigma$. Recall that $E_{K'}$ is diffeomorphic to the exterior of the dual curve $\gamma$ in $ S^3_0(K')$, with a diffeomorphism that identifies the meridian of $\gamma$ to the canonical longitude of $K'$. In other words, removing a tubular neighbourhood of $\gamma \cup \widehat{\Sigma}$ from $S^3_0(K')$ yields $M_\Sigma$, and therefore $M_\Sigma$ is diffeomorphic to the manifold obtained by removing from $N_{\widehat{\Sigma}}$ a tubular neighbourhood of the dual arc $\alpha$, and this, by Lemma  \ref{lem: cutting is cable}, is exactly $\overline{S^3 \setminus H_{\alpha}}$. This concludes the proof, since cutting along $\widehat{\Sigma}$ sends a meridian of $\gamma$ to a meridian of $\alpha$.
\end{proof}

\begin{figure}
    \centering
    \includegraphics[width=0.4\linewidth]{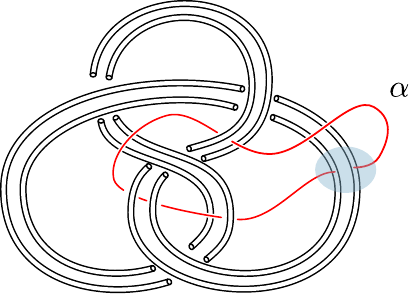}
    \caption{The picture shows the $(2,0)$-cable $L_K$ of $K$ in the case when $K$ is the figure-eight knot. The exterior of the arc $\alpha$ in $E_{L_K}$ is diffeomorphic to the manifold $M_{\Sigma}$. The shaded disc is relevant for Figure~\ref{fig:Delta}.}
    \label{fig:handlebody}
\end{figure}


For $i=0,1$, let $\Sigma_i$ be the once-punctured torus $\Sigma_i=\widehat{\Sigma}_i \cap \partial H_\alpha$.
By construction, the diffeomorphism $\widehat{\varphi}\colon \widehat{\Sigma}_0 \to \widehat{\Sigma}_1 $ sends one endpoint of $\alpha$ to the other, and therefore $\widehat{\varphi}$ restricts to a diffeomorphism $\varphi \colon \Sigma_0 \to \Sigma_1 $.

\begin{prop}\label{prop: gluing handlebody gives S3}
    Gluing $\Sigma_0 \times [0,1]$ to $\overline{S^3 \setminus H_{\alpha}}$  along $ \Sigma_0 \times \{0,1\}$ via the map 
    \[
    \psi = \id \times \{0\} \sqcup \varphi \times \{1\}
    \] yields $E_{K'}$. Moreover, there exists an extension $\widetilde{\psi}$ of $\psi$ on $ \partial (\Sigma_0 \times [0,1])$ such that gluing $\Sigma_0 \times [0,1] $ to $\overline{S^3 \setminus H_{\alpha}}$  via $ \widetilde{\psi}$ yields $S^3$. 
\end{prop}

\begin{proof}
We identify $M_{\Sigma}$ with $E_{K'}\setminus \big{(}\Sigma\times(0,1)\big{)}$. By Lemma \ref{lem: glue NSigma back}, we can assume that, for $i=0,1$, the diffeomorphism $\theta$ between $M_{\Sigma}$ and $\overline{S^3 \setminus H_{\alpha}}$ given by Proposition~\ref{prop: MSigma handlebody} maps $\Sigma\times \{i\}\subset \partial M_{\Sigma}$ to $\Sigma_i$ and satisfies $\theta(s,1)=\varphi \circ \theta(s,0)$ for all $s\in \Sigma$.

We identify $\Sigma\times [0,1]$ with $\Sigma_0\times [0,1]$ via the map $\theta_{|\Sigma\times \{0\}}\times \rm id$, and $M_{\Sigma}$ with $\overline{S^3 \setminus H_{\alpha}}$ via $\theta$. 
Gluing $M_\Sigma$ to $\Sigma\times [0,1]$ along $\id_{\Sigma \times \{0\}} \sqcup \id_{\Sigma \times \{1\}}$ produces $E_{K'}$, and this gluing can be extended to the whole boundary of $\Sigma \times [0,1]$ to produce $S^3$.
If we take the corresponding gluing between $\Sigma_0\times [0,1]$ and $\overline{S^3 \setminus H_{\alpha}}$ via the above identifications, we obtain the desired result.
\end{proof}

Consider the properly embedded, linearly independent arcs $\overset{\circ}{m}_{K_0}, \overset{\circ}{l}_{K_0}$ in $\Sigma_0$, corresponding to the meridian and the $n$-framed longitude of $K_0$. By crossing them with the interval $I$, they define compressing discs for the genus two handlebody $\Sigma_{0}\times I$, that we denote by $D_{m}$ and $D_l$ respectively. By Proposition \ref{prop: gluing handlebody gives S3}, we know that gluing a tubular neighborhood of these discs to $\overline{S^3 \setminus H_{\alpha}}$  via the restriction of $\widetilde{\psi}$ and then capping with a ball produces $S^3$. 

\begin{prop}\label{prop: Ktilde}
The manifold obtained by gluing a neighbourhood of the compressing disc $D_m$ to $\overline{S^3 \setminus H_{\alpha}}$  along the restriction of $\widetilde{\psi}$ is the exterior of a knot $\widetilde{K}\subset S^3$. Moreover $\widetilde{K}$ is a banding of $L_K$ along a band whose core is the dual arc $\alpha$.
\end{prop}

\begin{proof}
Orient the arc $\overset{\circ}{m}_{K_0}$ with the orientation induced by $m_{K_0}$, and $D_{m}=m_{K_0}\times I$ with the product orientation. Since $\psi(m_{K_0})=-m_{K_1}$, the image of $\partial D_m$, oriented with the boundary orientation, via $\widetilde{\psi}$ is given by 
\[\overset{\circ}{m}_{K_0} \cup j_1 \cup j_2 \cup  \overset{\circ}{m}_{K_1},
\] where $j_1, j_2$ are disjoint properly embedded arcs in the boundary of the tubular neighbourhood of $\alpha$, and where $\overset{\circ}{m}_{K_1}$ is the oriented arc in $\Sigma_1$ corresponding to $m_{K_1}$. 

There exists a disc $\Delta$, properly embedded in $H_\alpha$, that contains $\alpha$ and whose boundary is $\widetilde{\psi}(\partial D_m)$, see Figure \ref{fig:Delta}. 
\begin{figure}
    \centering
    \includegraphics[width=0.75\linewidth]{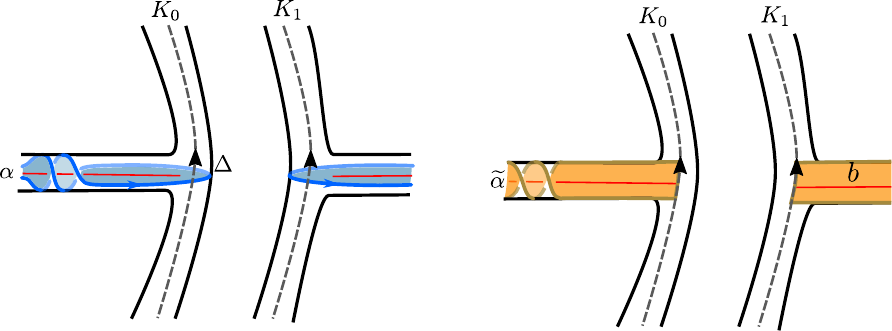}
    \caption{The figure shows the disc $\Delta$ and the band $b$ in $H_{\alpha}$, in a zoomed-in view of the region shaded in Figure~\ref{fig:handlebody}.}
    \label{fig:Delta}
\end{figure}
Gluing a neighbourhood of the compressing disc $D_m$ to $\overline{S^3 \setminus H_{\alpha}}$  coincides with gluing a neighbourhood of $\Delta$ to $S^3 \setminus H_\alpha$. 
Let $\widetilde{\alpha} \subseteq \Delta$ be an arc that extends $\alpha$ and connects $K_0$ to $K_1$.
Let $b$ be the band contained in $H_\alpha$ joining $K_0$ to $K_1$, whose core is $ \widetilde{\alpha}$, and that intersects $\Delta$ transversely and only in $\widetilde{\alpha}$. 
If $\widetilde{K}$ denotes the knot obtained by banding $L_K$ along $b$, then $H_\alpha$ minus a tubular neighbourhood of $ \Delta$ is a tubular neighbourhood of $\widetilde{K}$, thus concluding the proof.
\end{proof}
\begin{remark} \label{rmk: banding unoriented}
Recall that we oriented $L_K$ so that its components are parallel. With respect to this orientation, the banding from Proposition \ref{prop: Ktilde} is not oriented. However, if we change the orientation of one of the components, then it becomes an oriented banding. Notice that, with this orientation, $L_K$ is the boundary of a cabling annulus.
\end{remark}

\begin{prop}\label{prop: Ktilde is unknot}
    The manifold obtained by $+1$ surgery on $\widetilde{K}$ is diffeomorphic to $S^3$, via a diffeomorphism that maps the meridian of $\alpha$ to $K'$. In particular $\widetilde{K}$ is the unknot.
\end{prop}

\begin{proof}
By Propositions \ref{prop: gluing handlebody gives S3} and \ref{prop: Ktilde}, gluing a solid torus to the exterior of $\widetilde{K}$ along the slope defined by $\widetilde{\psi}(\partial D_l)$ gives $S^3$. 
By Proposition \ref{prop: MSigma handlebody} this diffeomorphism sends the meridian of $\alpha$ to $K'$.
Since ${\psi}(l_{K_0})= m_{K_1} + l_{K_1}$, where we are adopting the same orientation conventions of Proposition \ref{prop: Ktilde}, we have that $\widetilde{\psi}(\partial D_l)$ is equal to 
$$ \overset{\circ}{l}_{K_0} \cup k_1 \cup k_2 \cup -\overset{\circ} {(m_{K_1}+ l_{K_1})}$$
where $k_1$ and $k_2$ are disjoint properly embedded arcs in the boundary of the tubular neighbourhood of $\alpha$ parallel to $j_1$ and $j_2$. 
Here \(\overset{\circ}{(m_{K_1}+ l_{K_1})} \) indicates the properly embedded oriented arc in $\Sigma_1$ corresponding to the oriented simple closed curve representing $m_{K_1}+ l_{K_1}$. Clearly the simple closed curve $\widetilde{\psi}(\partial D_l)$ is a longitude of $\widetilde{K}$ and therefore defines an integer framing of it, which we now determine. 

First, we orient $\widetilde{K}$ by extending the orientation of $K_0$ (equivalently, by extending the orientation of  $-K_1$).
Note that the curve $\widetilde{\psi}(\partial D_l)$ is homologous to $l_{K_0}-(m_{K_1}+l_{K_1})$ in $\overline{S^3\setminus H_\alpha}$. The inclusion of $\overline{S^3\setminus H_\alpha}$ into $E_{L_K}$ induces an isomorphism on the first homology groups, and, since $l_{K_0}$ and $l_{K_1}$ are isotopic in $E_{L_K}$, one has the equality $l_{K_0}-(m_{K_1}+l_{K_1})= -m_{K_1}$ in $H_1(\overline{S^3\setminus H_\alpha})$. Since $-m_{K_1}$ is a positive meridian for $-K_1$, its image in $E_{\widetilde{K}}$ is a positive meridian of $\widetilde{K}$.
As a consequence of \cite{GorLue}, $\widetilde{K}$ must be the unknot.
\end{proof}

\begin{figure}[]
    \centering
\includegraphics[width=0.7\textwidth]{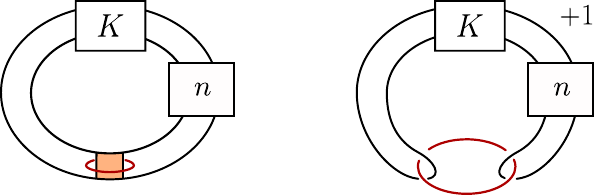}
    \caption{On the left: the band $b$ isotoped in a cabling annulus. On the right: The knot $K'$ is the image of the red curve in the $1$-surgery on $\widetilde{K}$. The box indicates a region containing $n$ right-handed full twists.}
    \label{fig: bing double}
\end{figure}

We are now ready to complete the proof of Theorem \ref{thm: main thm gen}.

\begin{proof}[Proof of Theorem \ref{thm: main thm gen}]
We start by noting that if $K$ is the unknot and $n=0$, then $W^n_K$ is the unknot, and the result follows from \cite[Corollary~8.3]{Gabaifol3}.

We can now assume that $K$ and $K'$ satisfy Assumption \ref{assum 2}.
In this case, by Proposition~\ref{prop: Ktilde is unknot} we know that $K'$ is the image of the meridian of the band $b$ in the $1$-surgery on $\widetilde{K}$, where $\widetilde{K}$ is the result of a banding of  the $(2,2n)$-cable $L_K$ of $K$ along the band $b$. By Remark~\ref{rmk: banding unoriented}, this is an oriented banding with respect to the orientation on $L_K$ induced by a cabling annulus. Hence, Theorem~\ref{thm: band in annulus} implies that the band $b$ can be isotoped in a cabling annulus $A$ in the exterior of $L_K$, and so we are in the situation described in Figure~\ref{fig: bing double}. Since the two components of the link in figure can be exchanged with an isotopy, we conclude that $K'$ is isotopic to $W_K^n$. 
\end{proof}

\bibliographystyle{amsalpha}
\bibliography{bibliography}

\providecommand{\bysame}{\leavevmode\hbox to3em{\hrulefill}\thinspace}
\providecommand{\MR}{\relax\ifhmode\unskip\space\fi MR }
\providecommand{\MRhref}[2]{%
  \href{http://www.ams.org/mathscinet-getitem?mr=#1}{#2}
}
\providecommand{\href}[2]{#2}
\begin{thebibliography}{KMOS07}

\bibitem[BS24a]{BStorusknot}
John~A Baldwin and Steven Sivek, \emph{Characterizing slopes for $5_2$}, Journal of the London Mathematical Society \textbf{109} (2024), no.~6, e12951.

\bibitem[BS24b]{BSinfinite}
John~A. Baldwin and Steven Sivek, \emph{Zero-surgery characterizes infinitely many knots}, Math. Res. Lett. \textbf{31} (2024), no.~6, 1639--1653. \MR{4862352}

\bibitem[BS25]{BSnonfib}
John~A Baldwin and Steven Sivek, \emph{Floer homology and non-fibered knot detection}, Forum of Mathematics, Pi, vol.~13, Cambridge University Press, 2025, p.~e1.

\bibitem[Bud05]{Budney2005JSJdecompositionsOK}
Ryan Budney, \emph{{JSJ-decompositions of knot and link complements in the 3-sphere}}, arXiv: Geometric Topology (2005).

\bibitem[EMn92]{Munoz}
Mario Eudave Mu\~noz, \emph{Band sums of links which yield composite links. {T}he cabling conjecture for strongly invertible knots}, Trans. Amer. Math. Soc. \textbf{330} (1992), no.~2, 463--501. \MR{1112545}

\bibitem[EN85]{EisNeu}
David Eisenbud and Walter~D. Neumann, \emph{{Three-dimensional link theory and invariants of plane curve singularities}}, Annals of Mathematics Studies, vol. 110, Princeton University Press, Princeton, 1985.

\bibitem[Gab87a]{gabaifol2}
David Gabai, \emph{{Foliations and the topology of 3-manifolds. II}}, Journal of Differential Geometry \textbf{26} (1987), no.~3, 461--478.

\bibitem[Gab87b]{Gabaifol3}
\bysame, \emph{{Foliations and the topology of 3-manifolds. III}}, Journal of Differential Geometry \textbf{26} (1987), no.~3, 479--536.

\bibitem[Gab87c]{gabai1987genus}
\bysame, \emph{Genus is superadditive under band connected sum}, Topology \textbf{26} (1987), no.~2, 209--210.

\bibitem[GL89]{GorLue}
Cameron~McA Gordon and John Luecke, \emph{Knots are determined by their complements}, Journal of the American Mathematical Society \textbf{2} (1989), no.~2, 371--415.

\bibitem[Hat23]{hatcher2007notes}
Allen Hatcher, \emph{Notes on basic 3-manifold topology}, Freely available at the website \url{https://pi.math.cornell.edu/~hatcher} (Version of 2023).

\bibitem[HPW25]{HPW25}
Kyle Hayden, Lisa Piccirillo, and Laura Wakelin, \emph{Dehn surgery functions are never injective}, arXiv preprint arXiv:2508.13369 (2025).

\bibitem[IBKR26]{K3}
R.~\.{I}nan\c{c} Baykur, Robion~C. Kirby, and Daniel Ruberman (eds.), \emph{K3---a new problem list in low-dimensional topology}, Mathematical Surveys and Monographs, vol. 295, American Mathematical Society, Providence, RI, [2026] \copyright 2026. \MR{5077826}

\bibitem[Jac80]{jaco1980lectures}
William~H Jaco, \emph{Lectures on three-manifold topology}, no.~43, American Mathematical Soc., 1980.

\bibitem[Joh79]{Johannson}
Klaus Johannson, \emph{Homotopy equivalences of {$3$}-manifolds with boundaries}, Lecture Notes in Mathematics, vol. 761, Springer, Berlin, 1979. \MR{551744}

\bibitem[JS79]{JacoShalen}
William~H. Jaco and Peter~B. Shalen, \emph{Seifert fibered spaces in {$3$}-manifolds}, Mem. Amer. Math. Soc. \textbf{21} (1979), no.~220, viii+192. \MR{539411}

\bibitem[Kir97]{kirby1997problems}
Rob Kirby, \emph{Problems in low-dimensional topology}, Proceedings 1993 Georgia International Topology Confence, Geometric Topology \textbf{2} (1997).

\bibitem[KM78]{KM78}
Robion Kirby and Paul Melvin, \emph{{Slice knots and property R}}, Inventiones mathematicae \textbf{45} (1978), no.~1, 57--59.

\bibitem[KMOS07]{KMOS07}
Peter Kronheimer, Tomasz Mrowka, Peter Ozsv{\'a}th, and Zolt{\'a}n Szab{\'o}, \emph{Monopoles and lens space surgeries}, Annals of mathematics (2007), 457--546.

\bibitem[KS26]{kegel2026search}
Marc Kegel and Jonathan Spreer, \emph{The search for exotic knot traces}, arXiv preprint arXiv:2603.22438 (2026).

\bibitem[Lac19]{Lac19}
Marc Lackenby, \emph{Every knot has characterising slopes}, Mathematische Annalen \textbf{374} (2019), no.~1, 429--446.

\bibitem[Mar25]{Martelli2016AnIT}
Bruno Martelli, \emph{An introduction to geometric topology}, Freely available at the website \url{https://people.dm.unipi.it/martelli/geometric_topology.html} (Version of September 2025).

\bibitem[OS19]{OSDehnsur19}
Peter Ozsv{\'a}th and Zolt{\'a}n Szab{\'o}, \emph{{The Dehn surgery characterization of the trefoil and the figure eight knot}}, Journal of Symplectic Geometry \textbf{17} (2019), no.~1, 251--265.

\bibitem[Pic20]{Pic20}
Lisa Piccirillo, \emph{The {Conway} knot is not slice}, Annals of Mathematics \textbf{191} (2020), no.~2, 581--591.

\bibitem[San24]{SanWhit}
Diego Santoro, \emph{{L--spaces, taut foliations and the Whitehead link}}, Algebraic \& Geometric Topology \textbf{24} (2024), no.~6, 3455--3502.

\bibitem[Wak26]{Wak}
Laura Wakelin, \emph{{Characterising slopes for hyperbolic knots and Whitehead doubles}}, Algebraic \& Geometric Topology \textbf{26} (2026), no.~2, 625--657.

\bibitem[Wal68]{WALDHAUSEN1967}
Friedhelm Waldhausen, \emph{Eine {K}lasse von 3-dimensionalen {M}annigfaltigkeiten. {II}}, Inventiones mathematicae \textbf{4} (1967/68), 87--117 (in German).

\end{thebibliography}

\end{document}